\documentclass[11pt,twoside]{amsart}

\usepackage[T1]{fontenc}
\usepackage[utf8]{inputenc}
\usepackage{mathpazo}
\usepackage{amssymb,amsmath,amsthm}
\usepackage{mathtools}
\usepackage{indentfirst}
\usepackage{url}
\usepackage[noadjust]{cite}
\usepackage{microtype}

\allowdisplaybreaks[4]

\usepackage[colorlinks=true,linkcolor=blue,anchorcolor=red,
citecolor=blue,urlcolor=blue,linktocpage,pdfpagelabels,
bookmarksnumbered,bookmarksopen]{hyperref}

\usepackage{mathtools}
\usepackage{xcolor,cancel}

\makeatletter
\let\FL@setaddresses\@setaddresses
\renewcommand{\@setaddresses}{%
  \par\addvspace{\bigskipamount}\noindent
  \begin{minipage}{\textwidth}
    \FL@setaddresses
  \end{minipage}\par
}
\makeatother

\newtheorem{Th}{Theorem}[section]

\newtheorem{Lem}[Th]{Lemma}
\newtheorem{Cor}[Th]{Corollary}

\theoremstyle{definition}
\newtheorem{Rem}[Th]{Remark}
\theoremstyle{plain}

\newtheorem{Que}[Th]{Question}

\newcommand{\R}{\mathbb{R}}

\newcommand{\C}{\mathbb{C}}

\numberwithin{equation}{section}

\newcommand{\loc}{\mathrm{loc}}

\newcommand{\Sph}{\mathbb{S}}
\newcommand{\Tclass}{\mathcal{T}}
\newcommand{\Xs}{\mathcal{X}_s}
\newcommand{\norm}[2]{\lVert #1\rVert_{#2}}
\DeclareFontEncoding{FMS}{}{}
\DeclareFontSubstitution{FMS}{futm}{m}{n}
\DeclareFontEncoding{FMX}{}{}
\DeclareFontSubstitution{FMX}{futm}{m}{n}
\DeclareSymbolFont{fouriergauge}{FMS}{futm}{m}{n}
\DeclareSymbolFont{fouriergaugelarge}{FMX}{futm}{m}{n}
\DeclareMathDelimiter{\VERT}
  {\mathord}{fouriergauge}{152}{fouriergaugelarge}{147}
\newcommand{\gnorm}[2]{\VERT #1\VERT_{#2}}
\newcommand{\abs}[1]{\lvert #1\rvert}
\newcommand{\dd}{\,\mathrm{d}}
\newcommand{\exd}{\mathrm d}
\newcommand{\ip}[2]{\langle #1,#2\rangle}

\title[Explicit counterexamples to Frank--Loss conjectures]{Explicit counterexamples to Frank--Loss optimizer conjectures
for critical Dirac and vector Sobolev inequalities}

\author[B. Bieganowski]{Bartosz Bieganowski}
\author[J. Mederski]{Jarosław Mederski}

\address{B. Bieganowski\newline\indent
	Faculty of Mathematics, Informatics and Mechanics, \newline\indent
	University of Warsaw, \newline\indent
	ul. Banacha 2, 02-097 Warsaw, Poland}
\email{\href{mailto:bartoszb@mimuw.edu.pl}{bartoszb@mimuw.edu.pl}}

\address{J. Mederski\newline\indent
    Institute of Mathematics, \newline\indent
    Polish Academy of Sciences, \newline\indent
    ul.~Śniadeckich 8, 00-656 Warsaw, Poland, \newline\indent
    and \newline\indent
	Faculty of Mathematics and Computer Science, \newline\indent
	Nicolaus Copernicus University, \newline\indent
	ul. Chopina 12/18, 87-100 Toruń, Poland}
\email{\href{mailto:jmederski@impan.pl}{jmederski@impan.pl}}


\begin{document}

\begin{abstract}
Frank and Loss proposed critical Sobolev inequalities for the Dirac operator
and for vector fields. They conjectured optimality of their spinorial family
and wrote that it is conceivable that the higher-dimensional Loss--Yau
potential is an optimizer. Wang and Zhang recently disproved both optimizer
claims using direct comparison and instability arguments. We compute
closed-form quotients for two elementary families of explicit competitors
with strictly smaller quotients. For every \(d\geq2\), the spinorial competitor is a constant spinor
multiplied by a scalar Aubin--Talenti optimizer. For every \(d\geq3\), the
vector competitor is the planar azimuthal field
\[
 A_d(x)=\frac{d}{(1+\lvert x\rvert^2)^2}
 (-x_2,x_1,0,\ldots,0).
\]
It lies in the nonlinear gauge and its quotient is computed in closed form.
For every odd \(d\geq3\), it gives an explicit proof of nonoptimality of
the higher-dimensional Loss--Yau--Dunne--Min potential. In dimension three,
this field is a specific member of the competitor family of Wang and Zhang:
it is a linear combination of the Loss--Yau potential and its reflected
twin. We also give a negative answer to Wang and Zhang's
invariant-class minimization question in dimension three, disproving
the conjectured minimality of the Loss--Yau potential within the
\(\mathrm{SO}(2)\times\{1\}\)-invariant class.

\end{abstract}

\maketitle

\section{Introduction and main results}

Let
\[
 \boldsymbol{\sigma}=(\sigma_1,\sigma_2,\sigma_3),
 \qquad
 \sigma_1=\begin{pmatrix}0&1\\1&0\end{pmatrix},\quad
 \sigma_2=\begin{pmatrix}0&-i\\i&0\end{pmatrix},\quad
 \sigma_3=\begin{pmatrix}1&0\\0&-1\end{pmatrix},
\]
be the Pauli matrices. Frank and Loss considered the optimal constants in 
\begin{equation}\label{eq:spin-ineq}
 \norm{\boldsymbol{\sigma}\cdot(-i\nabla)\psi}{3/2}
 \geq C_s\norm{\psi}{3}
\end{equation}
and in the \emph{Sobolev--curl inequality}
\begin{equation}\label{eq:vector-ineq-gauge}
 C_v\norm{A}{3}
 \leq \norm{\nabla\wedge A}{3/2},
 \qquad \nabla\cdot(\abs{A}A)=0,
\end{equation}
where all norms are taken over $\R^3$; see
\cite{FrankLossZero,FrankLossVector}. Here $\nabla\wedge A$ denotes the curl of $A$.
All derivatives are understood distributionally. The natural spinorial
domain and the corresponding optimal constant are
\begin{equation*}
\begin{aligned}
 \Xs&:=\{\psi\in L^3(\R^3;\C^2):
 \boldsymbol{\sigma}\cdot(-i\nabla)\psi\in L^{3/2}(\R^3;\C^2)\},\\
 C_s&:=\inf_{\psi\in\Xs\setminus\{0\}}
 \frac{\norm{\boldsymbol{\sigma}\cdot(-i\nabla)\psi}{3/2}}
      {\norm{\psi}{3}}.
\end{aligned}
\end{equation*}
The nonlinear divergence condition in \eqref{eq:vector-ineq-gauge} fixes the gauge naturally associated with the $L^3$ norm. We follow the notation $C_s,C_v$ from \cite{FrankLossVector}. 

Let
\begin{equation*}
 Y:=\{A\in L^3(\R^3;\R^3):\nabla\wedge A\in L^{3/2}(\R^3;\R^3)\}
\end{equation*}
and define the \emph{gauge seminorm}
\begin{equation*}
 \gnorm{A}{3}:=
 \inf_{\phi\in\dot W^{1,3}(\R^3)}\norm{A-\nabla\phi}{3}.
\end{equation*}
For every $n\geq2$ and $1<q<\infty$, $\dot W^{1,q}(\R^n)$ denotes the real-valued locally
integrable functions with weak gradient in $L^q(\R^n)$, modulo constants.
The minimization problem used by Frank and Loss is
\begin{equation}\label{eq:FL-S}
 \mathcal Q(A):=
 \frac{\norm{\nabla\wedge A}{3/2}^{3/2}}
      {\gnorm{A}{3}^{3/2}},
 \qquad
 S:=\inf_{\substack{A\in Y\\ \nabla\wedge A\neq0}}\mathcal Q(A),
 \qquad C_v=S^{2/3}.
\end{equation}
The subspace of $L^3$ gradients is closed: an $L^3$ limit of gradients is
curl-free distributionally and hence, by the distributional Poincar\'e lemma
on the simply connected space $\R^3$, is again the gradient of a
locally integrable function. A closed convex subspace of the reflexive space
$L^3$ admits a best approximation, and strict convexity makes it unique, see \cite{MederskiSzulkin}.
Thus there is a minimizing representative $A-\nabla\phi_A$ in every gauge
class. Its
Euler--Lagrange equation is
$\nabla\cdot(\abs{A-\nabla\phi_A}(A-\nabla\phi_A))=0$; conversely,
Lemma~\ref{lem:gauge} characterizes such a representative as minimizing.
Thus the gauge formulation in \eqref{eq:FL-S} is equivalent to
\eqref{eq:vector-ineq-gauge}; see also \cite{FrankLossVector}.
The infimum of the gauge quotient is unchanged if one restricts to smooth
compactly supported vector fields, by the density argument in
Section~\ref{sec:FL-vector-all-d}.
In the notation of \cite{MederskiSzulkin},
$S=S_{3/2,\mathrm{curl}}$. The Sobolev--curl framework for the quadratic
curl--curl problem, corresponding to $p=2$, was for the first time developed in
\cite{MederskiSzulkin2021} and later extended to the $p$-curl setting with $1<p<3$, hence
including $p=3/2$, in \cite{MederskiSzulkin}.

The appearance of zero modes originates in the stability
theory of Coulomb systems coupled to a classical magnetic field. Fr\"ohlich,
Lieb, and Loss studied the one-electron atom with the magnetic-field
self-energy included and established stable and unstable regimes separated by a
critical nuclear charge \cite{FrohlichLiebLoss}. Loss and Yau subsequently
constructed magnetic fields of finite self-energy for which the Pauli operator
has a zero-energy bound state; by a scaling argument, such a zero mode leads to
the collapse of the corresponding one-electron atom for sufficiently large
nuclear charge \cite{LossYau}. Their explicit construction is the source of the
following distinguished pair.
The {\em Loss--Yau zero-mode pair} is given by
\begin{equation}\label{eq:LY-spinor}
 \psi_{\mathrm{LY}}(x)
 =\frac{1+i\boldsymbol{\sigma}\cdot x}{(1+\abs{x}^2)^{3/2}}\eta
\end{equation}
and
\begin{equation}\label{eq:LY-potential}
 A_{\mathrm{LY}}(x)
 =\frac{3}{(1+\abs{x}^2)^2}
 ((1-\abs{x}^2)w+2(x\cdot w)x+2w\wedge x),
\end{equation}
where $\eta\in\C^2$ is a constant spinor with $\abs{\eta}=1$ and
$w=\ip{\eta}{\boldsymbol{\sigma}\eta}$; in particular, $\abs{w}=1$.
In Remark~9.1 of \cite{FrankLossZero}, Frank and Loss conjectured that this
spinor and vector potential are optimizers for \eqref{eq:spin-ineq} and
\eqref{eq:vector-ineq-gauge}; see also \cite{FrankLossVector,FrankLossSharp}. The
corresponding quotient values for the
 Loss--Yau pair are
\begin{equation}\label{eq:FL-candidates}
 C_s^{\mathrm{LY}}
 =\frac32\abs{\Sph^3}^{1/3},
 \qquad
 C_v^{\mathrm{LY}}
 =2\abs{\Sph^3}^{1/3},
 \qquad \abs{\Sph^3}=2\pi^2,
\end{equation}
where $\Sph^3$ is the unit sphere in $\mathbb{R}^4$.

We shall also use the axisymmetric azimuthal class $\Tclass$ introduced in
\cite{MederskiSzulkin} for the Sobolev--curl inequality.
Write $(\rho,\theta,z)$ for cylindrical coordinates and
$(e_\rho,e_\theta,e_3)$ for the corresponding orthonormal frame.
The class consists of fields
\[
 A=a(\rho,z)e_\theta=\frac{a(\rho,z)}{\rho}(-x_2,x_1,0),
 \qquad \rho := \sqrt{x_1^2+x_2^2}>0.
\]
Every $A\in Y\cap\Tclass$ satisfies
$\nabla\cdot(\abs{A}A)=0$ distributionally. Indeed, for
$\varphi\in C_c^\infty(\R^3)$, cylindrical coordinates give
\[
 \int_{\R^3}\abs{A}A\cdot\nabla\varphi\dd x
 =\int_{\R}\int_0^\infty\abs{a}a
   \left(\int_0^{2\pi}\partial_\theta\varphi\dd\theta\right)
   \dd\rho\dd z=0.
\]
Fubini's theorem applies because
$\abs{\partial_\theta\varphi}\leq\rho\norm{\nabla\varphi}{\infty}$
and $A\in L^3$ implies $\abs{A}^2\in L^1_{\loc}$.
Hence Lemma~\ref{lem:gauge} gives
$\gnorm{A}{3}=\norm{A}{3}$ throughout this class.
This is one of the fixed-point classes considered in
\cite{MederskiSzulkin}.
Let $C_v^{\Tclass}$ denote the best constant in \eqref{eq:vector-ineq-gauge} restricted to this class; equivalently,
\begin{equation*}
 S_{3/2,\mathrm{curl}}^{\Tclass}
 :=\inf_{\substack{A\in Y\cap\Tclass\\ \nabla\wedge A\neq0}}
 \frac{\norm{\nabla\wedge A}{3/2}^{3/2}}
      {\gnorm{A}{3}^{3/2}},
 \qquad
 C_v^{\Tclass}
 =(S_{3/2,\mathrm{curl}}^{\Tclass})^{2/3}.
\end{equation*}
By \cite{MederskiSzulkin}, the restricted infimum
$S_{3/2,\mathrm{curl}}^{\Tclass}$ is attained by some nonzero field in
$Y\cap\Tclass$. We will not use this existence result; in
Section~\ref{sec:azimuthal-field} we construct a fully explicit competitor
in this class.

Our main statement is the following.

\begin{Th}\label{thm:main}
The optimal constants satisfy
\begin{align}
 C_s
 &\leq (18\pi)^{1/3}
 =\left(\frac{8}{3\pi}\right)^{1/3}
   \frac32\abs{\Sph^3}^{1/3}
 <\frac32\abs{\Sph^3}^{1/3},
 \label{eq:main-Cs}\\
 C_v
 &\leq C_v^{\Tclass}
 \leq 2\left(\frac{80\pi^2}{49}\right)^{1/3}
 =2\left(\frac{40}{49}\right)^{1/3}\abs{\Sph^3}^{1/3}
 <2\abs{\Sph^3}^{1/3}.
 \label{eq:main-Cv}
\end{align}
Equivalently,
\begin{equation}\label{eq:main-power}
 S_{3/2,\mathrm{curl}}
 \leq S_{3/2,\mathrm{curl}}^{\Tclass}
 \leq \frac{8\sqrt{10}}7\pi
 <4\pi.
\end{equation}
Consequently, neither $\psi_{\mathrm{LY}}$ nor $A_{\mathrm{LY}}$ is a global
optimizer of its respective inequality.
\end{Th}

The theorem does not determine the exact values of the optimal constants.
It is consistent with the existence and compactness results in
\cite{FrankLossVector}, as well as with Theorem~1 and the classification of
equality cases in \cite{FrankLossSharp}. Wang and Zhang recently proved the
spinorial nonoptimality in every $d\geq2$ \cite{WangZhangSpinor} and the
vector-potential nonoptimality in every odd $d\geq3$
\cite[Theorem~1.3]{WangZhangConstants}. In particular, in dimension three,
$C_s<\frac32\abs{\Sph^3}^{1/3}$ and
$C_v<2\abs{\Sph^3}^{1/3}$. Our contribution is to provide two elementary
explicit families with closed-form quotients and quantitative upper bounds
for $C_s$, $C_v$ and $S_{3/2,\mathrm{curl}}$.
We also give a negative answer to Wang and Zhang's invariant-class
minimization question in dimension three
\cite{WangZhangConstants}, contrary to their conjectured affirmative
answer. More precisely, the Loss--Yau potential fails to minimize the
gauge quotient even within the class of axisymmetric fields with the
same distinguished axis.
Section~\ref{sec:FL-spinor-all-d}
treats the spinorial conjecture for every $d\geq2$, while
Section~\ref{sec:FL-vector-all-d} constructs a vector competitor for every
$d\geq3$ and disproves the higher-dimensional vector conjecture in every odd
dimension.

Theorem~1.3 of \cite{WangZhangConstants} establishes vector-potential
nonoptimality for the $1$-form problem with exponent $d/2$ in every odd
dimension by a second-variation argument on the sphere. In dimension three,
this problem also coincides with their middle-degree curl problem.
In \cite[Proposition~5.1]{WangZhangConstants} the authors give a direct comparison for the gauge quotient:
combining positive and negative first curl eigenforms and applying
H\"older's inequality produces explicit competitors with smaller quotients.
Conformal invariance transfers this comparison to $\R^3$.
In the notation of \cite[formula~(1.4)]{WangZhangConstants}, specialized
to dimension three, for $A\in Y$ with
$\nabla\wedge A\neq0$ the gauge quotient is
\begin{equation}\label{eq:WZ-J2}
 J_2(A)
 :=\frac{\norm{\nabla\wedge A}{3/2}^{2}}{\gnorm{A}{3}^{2}}
 =\mathcal Q(A)^{4/3}.
\end{equation}
Consequently, their strict inequality for $J_2$ is equivalent to
$C_v<2\abs{\Sph^3}^{1/3}$, or, equivalently, to
$S_{3/2,\mathrm{curl}}<4\pi$. The non-strict inequality
$S_{3/2,\mathrm{curl}}\leq4\pi$ was already known from
\cite{MederskiSzulkin}.

At the end of \cite{WangZhangConstants}, Wang and Zhang formulate an open
question and state that they believe the answer is affirmative. Their exact
wording is \cite[p.~24]{WangZhangConstants}: \emph{Are they minimizers of
$J_2$ within the class of invariant vector fields?} Here the pronoun refers
to the Loss--Yau solutions constructed in their Section~6.1.
The group and its action are specified in their formula~(6.45):
$g\in\{1\}\times\mathrm{SO}(n-1)$ acts by
$u_g(x)=g\,u(g^{-1}x)$. For $n=3$, these are rotations about the
distinguished axis, which we rotate from $e_1$ to $e_3$ below.

We now give a precise dimension-three formulation. After a rotation, take
$w=e_3$ in \eqref{eq:LY-potential}, and let
\[
 G:=\{R\in\mathrm{SO}(3):Re_3=e_3\}\simeq\mathrm{SO}(2),
 \qquad
 (R\cdot A)(x):=R A(R^{-1}x).
\]
Since rotations preserve scalar and vector products, for $w=e_3$ and
$R\in G$ formula~\eqref{eq:LY-potential} gives
\begin{equation}\label{eq:LY-equivariance}
 A_{\mathrm{LY}}(Rx)=R A_{\mathrm{LY}}(x),
\end{equation}
so the Loss--Yau potential is fixed by this action.

Define
\[
 Y^G
 :=\bigl\{A\in Y:\nabla\wedge A\neq0,\ 
       R\cdot A=A\ \text{for every }R\in G\bigr\}.
\]

\begin{Que}[Three-dimensional formulation of the Wang--Zhang question]
\label{question:WZ-symmetric}
Does the Loss--Yau potential minimize $J_2$ in the invariant class $Y^G$?
Equivalently,
\[
 J_2(A_{\mathrm{LY}})
 \stackrel{?}{=}
 \inf_{A\in Y^G}J_2(A).
\]
\end{Que}

We answer this question negatively, thereby disproving
the conjectured minimality in the invariant class $Y^G$.
The axisymmetric azimuthal field
\[
 A_{\Tclass}(x):=\frac{3}{(1+\abs{x}^2)^2}(-x_2,x_1,0)
\]
belongs to $Y^G$ and satisfies
\[
 \inf_{A\in Y^G}J_2(A)
 \leq J_2(A_{\Tclass})
 =\left(\frac{40}{49}\right)^{2/3}J_2(A_{\mathrm{LY}})
 <J_2(A_{\mathrm{LY}}).
\]

In dimension three, let $J_{3/2,1}$ denote the functional introduced in \cite[Section~6]{WangZhangConstants}, restricted to the nonlinear-gauge
slice. Then
\[
 J_2=J_{3/2,1}^{2/3}.
\]
Thus, on the nonlinear-gauge slice, the two functionals agree up to a
positive power and therefore have the same minimizers.
The relation of our three-dimensional competitor to the two Killing forms
used by Wang and Zhang is made explicit in
\eqref{eq:AT-reflected-twin} below.

\section{Gauge minimization}

We recall the standard convexity characterization of the minimizing
representative in a gauge class, stated here for general $n$ and $q$;
compare the converse assertion in \cite[Lemma~2.3]{FrankLossVector}
for $n=q=3$.

\begin{Lem}\label{lem:gauge}
Let $n\geq2$, $1<q<\infty$, and $A\in L^q(\R^n;\R^n)$. If
\begin{equation}\label{eq:nonlinear-gauge-q}
 \nabla\cdot(\abs{A}^{q-2}A)=0
\end{equation}
in the sense of distributions, then
\[
 \inf_{\phi\in\dot W^{1,q}(\R^n)}\norm{A-\nabla\phi}{q}
 =\norm{A}{q}.
\]
\end{Lem}

\begin{proof}
Consider the convex differentiable functional
\[
 F(\phi):=\frac1q\norm{A-\nabla\phi}{q}^q,
 \qquad \phi\in\dot W^{1,q}(\R^n).
\]
Since $\abs{A}^{q-2}A\in L^{q/(q-1)}$, the distributional condition
\eqref{eq:nonlinear-gauge-q}, H\"older's inequality, and the standard
density of $C_c^\infty(\R^n)$ in $\dot W^{1,q}(\R^n)$ with respect to
the gradient norm imply
\[
 F'(0)[\phi]
 =-\int_{\R^n}\abs{A}^{q-2}A\cdot\nabla\phi\dd x=0
 \qquad\text{for every }\phi\in\dot W^{1,q}(\R^n).
\]
Convexity now gives
\[
 F(\phi)\geq F(0)+F'(0)[\phi]=F(0).
\]
Thus the infimum is attained at $\phi=0$, which proves the assertion.
\end{proof}

\section{The spinorial trial function}

The key observation is that the restriction of \eqref{eq:spin-ineq} to spinors of constant direction contains the scalar critical Sobolev inequality for $p=3/2$.

\begin{Lem}
Let $\eta\in\C^2$ be a constant unit spinor and let
$0\neq u\in L^3(\R^3;\R)$ be weakly differentiable with
$\nabla u\in L^{3/2}(\R^3;\R^3)$. For $\psi=u\eta\in\Xs$ one has, almost
everywhere,
\begin{equation*}
 \abs{\boldsymbol{\sigma}\cdot(-i\nabla)\psi}=\abs{\nabla u}.
\end{equation*}
Consequently,
\begin{equation*}
 \frac{\norm{\boldsymbol{\sigma}\cdot(-i\nabla)\psi}{3/2}}
      {\norm{\psi}{3}}
 =\frac{\norm{\nabla u}{3/2}}{\norm{u}{3}}.
\end{equation*}
\end{Lem}

\begin{proof}
The Pauli relations give
\[
 (\boldsymbol{\sigma}\cdot a)^2=\abs{a}^2I_{\C^2},
 \qquad a\in\R^3.
\]
Therefore
\[
 \abs{\boldsymbol{\sigma}\cdot(-i\nabla)(u\eta)}^2
 =\ip{\eta}{(\boldsymbol{\sigma}\cdot\nabla u)^2\eta}
 =\abs{\nabla u}^2\abs{\eta}^2.
\]
\end{proof}

For the scalar Aubin--Talenti optimizer at $p=3/2$ we take
\cite{Aubin,Talenti}
\begin{equation}\label{eq:Talenti}
 u(x):=\frac{1}{1+\abs{x}^3}.
\end{equation}
Direct computations show that
\begin{align*}
 \norm{u}{3}^3
 &=4\pi\int_0^\infty\frac{r^2}{(1+r^3)^3}\dd r
 =\frac{2\pi}{3},
 \\
 \norm{\nabla u}{3/2}^{3/2}
 &=4\pi 3^{3/2}\int_0^\infty\frac{r^5}{(1+r^3)^3}\dd r
 =2\pi\sqrt3.
\end{align*}
It follows that
\begin{equation*}
 \frac{\norm{\boldsymbol{\sigma}\cdot(-i\nabla)(u\eta)}{3/2}}
      {\norm{u\eta}{3}}
 =\frac{(2\pi\sqrt3)^{2/3}}{(2\pi/3)^{1/3}}
 =(18\pi)^{1/3}.
\end{equation*}
Since
\[
 18\pi<\frac{27}{4}\pi^2
 =\left(\frac32\abs{\Sph^3}^{1/3}\right)^3,
\]
we obtain \eqref{eq:main-Cs}.

\section{The Frank--Loss spinorial conjecture in arbitrary dimension}
\label{sec:FL-spinor-all-d}

We first fix the special-function and measure conventions used below. For
$a,b>0$,
\begin{equation}\label{eq:gamma-beta-definitions}
 \Gamma(a):=\int_0^\infty t^{a-1}e^{-t}\dd t,
 \qquad
 \mathrm B(a,b):=\int_0^1t^{a-1}(1-t)^{b-1}\dd t
 =\frac{\Gamma(a)\Gamma(b)}{\Gamma(a+b)}.
\end{equation}
We shall also use the equivalent formula
\begin{equation}\label{eq:beta-infinite-interval}
 \mathrm B(a,b)=\int_0^\infty
 \frac{t^{a-1}}{(1+t)^{a+b}}\dd t.
\end{equation}
We recall also the recurrence
$\Gamma(a+1)=a\Gamma(a)$ and the sphere-area formula
\begin{equation}\label{eq:sphere-area}
 \abs{\Sph^k}=\frac{2\pi^{(k+1)/2}}{\Gamma((k+1)/2)}.
\end{equation}
In what follows $\dd\omega_{\Sph^k}$ denotes the Riemannian surface measure such that
$\int_{\Sph^k}1\dd\omega_{\Sph^k}=\abs{\Sph^k}$.

Let $d\geq2$, set $N=2^{\lfloor d/2\rfloor}$, and let
$\gamma_1,\ldots,\gamma_d$ be Hermitian matrices on $\C^N$ satisfying the
Clifford relations
\begin{equation}\label{eq:Clifford-relations}
 \gamma_j\gamma_k+\gamma_k\gamma_j=2\delta_{jk}I.
\end{equation}
Set
\begin{equation*}
 p_d:=\frac{2d}{d+1},
 \qquad q_d:=\frac{2d}{d-1}=p_d^*,
 \qquad D_d:=\boldsymbol{\gamma}\cdot(-i\nabla).
\end{equation*}
Inequality \cite[formula (9)]{FrankLossSharp} rewrites as
\begin{equation}\label{eq:FL-sharp-spinor-general}
 \norm{D_d\psi}{L^{p_d}}^{p_d}
 \geq S_d^s\norm{\psi}{L^{q_d}}^{p_d}
\end{equation}
with $S_d^s$ being a sharp constant characterized by
\begin{equation*}
 S_d^s=\inf_{\substack{\psi\in L^{q_d}(\R^d;\C^N)\setminus\{0\}\\
                         D_d\psi\in L^{p_d}(\R^d;\C^N)}}
 \mathcal Q_d(\psi),
\end{equation*}
where
\begin{equation*}
 \mathcal Q_d(\psi):=
 \frac{\norm{D_d\psi}{L^{p_d}}^{p_d}}
      {\norm{\psi}{L^{q_d}}^{p_d}}.
\end{equation*}
This domain is the closure of $C_c^\infty(\R^d;\C^N)$ in the norm
$\norm{\psi}{L^{q_d}}+\norm{D_d\psi}{L^{p_d}}$.
Indeed, $1/p_d=1/q_d+1/d$ and a standard cutoff $\chi_R$ give
\[
 \norm{D_d(\chi_R\psi)-D_d\psi}{L^{p_d}}
 \leq\norm{(1-\chi_R)D_d\psi}{L^{p_d}}
      +C\norm{\psi}{L^{q_d}(B_{2R}\setminus B_R)}\longrightarrow0,
\]
since $\norm{\nabla\chi_R}{L^d}$ is bounded uniformly in $R$.
Also $\chi_R\psi\to\psi$ in $L^{q_d}$, and mollification completes
the approximation.
For every $d\geq2$, \cite[formula (10)]{FrankLossSharp} gives the family
\begin{equation}\label{eq:FL-family-10-d}
 \Psi_{\phi_0,\phi_1}(x)
 :=\frac{\phi_0+\boldsymbol{\gamma}\cdot x\phi_1}
         {(1+\abs{x}^2)^{d/2}},
\end{equation}
where $\phi_0,\phi_1\in\C^N$ satisfy
\begin{equation}\label{eq:FL-family-conditions-d}
 \abs{\phi_0}=\abs{\phi_1}=:a>0,
 \qquad
 \operatorname{Re}\ip{\phi_0}{\gamma_j\phi_1}=0,
 \quad j=1,\ldots,d.
\end{equation}
For $d=3$, taking $\gamma_j=\sigma_j$, $\phi_0=\eta$, and
$\phi_1=i\eta$ gives precisely the Loss--Yau spinor in
\eqref{eq:LY-spinor}.
These spinors solve the Euler--Lagrange equation associated with
\eqref{eq:FL-sharp-spinor-general}. Immediately after defining \eqref{eq:FL-family-10-d}, Frank and Loss
conjecture that the sharp constant in \eqref{eq:FL-sharp-spinor-general} is attained by \eqref{eq:FL-family-10-d}, \cite[p.~311]{FrankLossSharp}. Wang and Zhang subsequently disproved
this conjecture. Their direct comparison in
\cite[Proposition~5.5 and Appendix~C]{WangZhangSpinor} gives a strict
competitor whose quotient lies below the common value of the conformal
candidate family.
Conformal invariance transfers the result to $\R^d$.
The radial profile used below is the classical scalar Aubin--Talenti
profile at $p=p_d$ \cite{Aubin,Talenti}, multiplied by a constant spinor.
The following theorem gives a different elementary
counterexample, with an explicit radial profile and an exact quotient, in
every dimension in which the conjecture was stated.

\begin{Th}\label{thm:FL-spinor-all-d}
For every integer $d\geq2$, define
\begin{align}
 I_d^{\mathrm{AT}}
 &:=\frac{\abs{\Sph^{d-1}}}{q_d}
 \mathrm B\left(\frac{d-1}{2},\frac{d+1}{2}\right),
 \notag\\
 I_d^{\mathrm{FL}}
 &:=\frac{\abs{\Sph^{d-1}}}{2}
 \mathrm B\left(\frac d2,\frac d2\right)
 =\frac{\pi^{d/2}\Gamma(d/2)}{\Gamma(d)}
 =\frac{\abs{\Sph^d}}{2^d}.
 \label{eq:I-FL-d}
\end{align}
Then the constant-direction spinor
\begin{equation}\label{eq:AT-spinor-d}
 \psi_d(x):=u_d(x)\eta,
 \qquad
 u_d(x):=(1+\abs{x}^{q_d})^{-(d-1)/2},
 \qquad \abs{\eta}=1,
\end{equation}
satisfies
\begin{equation}\label{eq:AT-quotient-d}
 \mathcal Q_d(\psi_d)
 =d^{p_d}(I_d^{\mathrm{AT}})^{2/(d+1)}.
\end{equation}
Every spinor in \eqref{eq:FL-family-10-d}--\eqref{eq:FL-family-conditions-d}
satisfies
\begin{equation}\label{eq:FL-quotient-d}
 \mathcal Q_d(\Psi_{\phi_0,\phi_1})
 =d^{p_d}(I_d^{\mathrm{FL}})^{2/(d+1)}.
\end{equation}
Moreover,
\begin{equation}\label{eq:integral-ratio-d}
 \frac{I_d^{\mathrm{AT}}}{I_d^{\mathrm{FL}}}
 =\frac{(d-1)^2}{2d}
 \left(\frac{\Gamma((d-1)/2)}{\Gamma(d/2)}\right)^2
 <1.
\end{equation}
Consequently,
\begin{equation}\label{eq:FL-disproof-d}
 S_d^s
 \leq d^{p_d}(I_d^{\mathrm{AT}})^{2/(d+1)}
 <d^{p_d}(I_d^{\mathrm{FL}})^{2/(d+1)},
\end{equation}
and no nonzero spinor in the family \cite[formula (10)]{FrankLossSharp} is an
optimizer of \eqref{eq:FL-sharp-spinor-general}.
\end{Th}

\begin{proof}
For a real-valued weakly differentiable function $v$ and a constant unit
spinor $\eta$, the Clifford relations give
\[
 \abs{D_d(v\eta)}^2
 =\ip{\eta}{(\boldsymbol{\gamma}\cdot\nabla v)^2\eta}
 =\abs{\nabla v}^2.
\]
For the function in \eqref{eq:AT-spinor-d},
\[
 u_d^{q_d}=(1+r^{q_d})^{-d},
 \qquad
 \abs{u_d'(r)}
 =d r^{q_d-1}(1+r^{q_d})^{-(d+1)/2}.
\]
For every radial integrand $F$, polar coordinates mean
\begin{equation}\label{eq:radial-integration-d}
 \int_{\R^d}F(\abs{x})\dd x
 =\abs{\Sph^{d-1}}\int_0^\infty F(r)r^{d-1}\dd r.
\end{equation}
In the first integral below, set $t=r^{q_d}$. Since
$r^{d-1}\dd r=q_d^{-1}t^{d/q_d-1}\dd t$ and
$d/q_d=(d-1)/2$, \eqref{eq:beta-infinite-interval} gives
\begin{align*}
 \int_{\R^d}u_d^{q_d}\dd x
 &=\frac{\abs{\Sph^{d-1}}}{q_d}
   \int_0^\infty
   \frac{t^{(d-1)/2-1}}{(1+t)^d}\dd t=\frac{\abs{\Sph^{d-1}}}{q_d}
   \mathrm B\left(\frac{d-1}{2},\frac{d+1}{2}\right)
   =I_d^{\mathrm{AT}},\\
 \int_{\R^d}\abs{\nabla u_d}^{p_d}\dd x
 &=\frac{d^{p_d}\abs{\Sph^{d-1}}}{q_d}
   \int_0^\infty
   \frac{t^{(d+1)/2-1}}{(1+t)^d}\dd t=\frac{d^{p_d}\abs{\Sph^{d-1}}}{q_d}
   \mathrm B\left(\frac{d+1}{2},\frac{d-1}{2}\right)
   =d^{p_d}I_d^{\mathrm{AT}}.
\end{align*}
For the second computation we used
$p_d(q_d-1)=q_d$ and $p_d(d+1)/2=d$. 
Since $1-p_d/q_d=2/(d+1)$, this proves
\eqref{eq:AT-quotient-d}.

For comparison, we compute the known quotient of the Frank--Loss family
in our normalization; see \cite[Appendix~C]{WangZhangSpinor}.
The identities \eqref{eq:Clifford-relations} and
\eqref{eq:FL-family-conditions-d} imply
\begin{align*}
 \abs{\phi_0+\boldsymbol{\gamma}\cdot x\phi_1}^2
 &=a^2(1+\abs{x}^2),\\
 \abs{\boldsymbol{\gamma}\cdot x\phi_0-\phi_1}^2
 &=a^2(1+\abs{x}^2).
\end{align*}
Consequently,
\begin{equation*}
 \abs{\Psi_{\phi_0,\phi_1}(x)}
 =a(1+\abs{x}^2)^{-(d-1)/2}.
\end{equation*}
A direct differentiation gives
\begin{equation*}
 D_d\Psi_{\phi_0,\phi_1}
 =i d(1+\abs{x}^2)^{-(d+2)/2}
   (\boldsymbol{\gamma}\cdot x\phi_0-\phi_1),
\end{equation*}
and hence
\begin{equation*}
 \abs{D_d\Psi_{\phi_0,\phi_1}(x)}
 =da(1+\abs{x}^2)^{-(d+1)/2}.
\end{equation*}
Using \eqref{eq:radial-integration-d} and the substitution $t=|x|^2$, we have
\begin{align*}
 \int_{\R^d}(1+\abs{x}^2)^{-d}\dd x
 =\frac{\abs{\Sph^{d-1}}}{2}
   \int_0^\infty\frac{t^{d/2-1}}{(1+t)^d}\dd t
 =\frac{\abs{\Sph^{d-1}}}{2}\mathrm B\left(\frac d2,\frac d2\right)
 =I_d^{\mathrm{FL}}.
\end{align*}
The first closed form in \eqref{eq:I-FL-d} follows from
\eqref{eq:gamma-beta-definitions} and \eqref{eq:sphere-area}. The second one
follows from Legendre's duplication formula
\begin{equation}\label{eq:gamma-duplication}
 \Gamma(z)\Gamma(z+\tfrac12)
 =2^{1-2z}\sqrt\pi\,\Gamma(2z).
\end{equation}
Indeed, applying it with $z=d/2$ gives
$\pi^{d/2}\Gamma(d/2)/\Gamma(d)=\abs{\Sph^d}/2^d$.
Consequently,
\[
 \int_{\R^d}\abs{\Psi_{\phi_0,\phi_1}}^{q_d}\dd x
 =a^{q_d}I_d^{\mathrm{FL}},
 \qquad
 \int_{\R^d}\abs{D_d\Psi_{\phi_0,\phi_1}}^{p_d}\dd x
 =d^{p_d}a^{p_d}I_d^{\mathrm{FL}}.
\]
This proves \eqref{eq:FL-quotient-d}.

For completeness, the equality in \eqref{eq:integral-ratio-d} follows by
writing both beta functions in gamma form:
\begin{align*}
 \frac{I_d^{\mathrm{AT}}}{I_d^{\mathrm{FL}}}
 =\frac{2}{q_d}
   \frac{\Gamma((d-1)/2)\Gamma((d+1)/2)}{\Gamma(d/2)^2}=\frac{(d-1)^2}{2d}
   \left(\frac{\Gamma((d-1)/2)}{\Gamma(d/2)}\right)^2.
\end{align*}
Here we used $2/q_d=(d-1)/d$ and
$\Gamma((d+1)/2)=((d-1)/2)\Gamma((d-1)/2)$.
To prove strictness, put
$x=(d-1)/2>0$. Strict log-convexity of the gamma function, applied at the
midpoint $x+1$ of $x+1/2$ and $x+3/2$, gives
\[
 \Gamma(x+1)^2
 <\Gamma(x+1/2)\Gamma(x+3/2)
 =(x+1/2)\Gamma(x+1/2)^2.
\]
Since $\Gamma(x+1)=x\Gamma(x)$, it follows that
\[
 \frac{\Gamma(x+1/2)}{\Gamma(x)}
 >\frac{x}{\sqrt{x+1/2}}.
\]
Indeed,
\[
 \frac{I_d^{\mathrm{AT}}}{I_d^{\mathrm{FL}}}
 =\frac{x^2}{x+1/2}
  \left(\frac{\Gamma(x)}{\Gamma(x+1/2)}\right)^2<1,
\]
which is the strict inequality in \eqref{eq:integral-ratio-d}. Both trial
families belong to the natural
spaces in \eqref{eq:FL-sharp-spinor-general}.
The conclusion \eqref{eq:FL-disproof-d} follows.
\end{proof}

\begin{Cor}
For $d=3$, one has
\[
 I_3^{\mathrm{AT}}=\frac{2\pi}{3},
 \qquad I_3^{\mathrm{FL}}=\frac{\pi^2}{4},
\]
and therefore
\begin{equation*}
 S_3^s
 \leq\sqrt{18\pi}
 <\frac{3\sqrt3}{2}\pi.
\end{equation*}

\end{Cor}

\section{The Frank--Loss vector conjecture in arbitrary dimension}
\label{sec:FL-vector-all-d}

Let $d\geq3$. Following Frank and Loss, we write $\nabla\wedge A$ for
the antisymmetric derivative of a vector field $A=(A_1,\ldots,A_d)$,
with independent coefficients and pointwise norm
\begin{equation}\label{eq:antisym-gradient-convention}
\begin{aligned}
 (\nabla\wedge A)_{ij}
 &:=\partial_iA_j-\partial_jA_i,\qquad 1\leq i<j\leq d,\\
 \abs{\nabla\wedge A}^2
 &:=\sum_{1\leq i<j\leq d}
   \abs{\partial_iA_j-\partial_jA_i}^2.
\end{aligned}
\end{equation}
Equivalently, these are the coefficients of the two-form
$\exd(A^\flat)=\sum_{i<j}(\partial_iA_j-\partial_jA_i)
\dd x_i\wedge\dd x_j$, where $A^\flat:=\sum_iA_i\dd x_i$.
In dimension three we identify this two-form with the usual curl vector
by the Hodge star; the identification preserves its norm. For $d>3$,
$\nabla\wedge A$ has $\binom d2$ independent components and is not a
vector in $\R^d$. This normalization agrees with
\cite[Theorem~1.3]{FrankLossZero} and \cite[formula~(8)]{FrankLossSharp}.
We set
\begin{align*}
 Y_d&:=\{A\in L^d(\R^d;\R^d):
       \nabla\wedge A\in L^{d/2}(\R^d;\R^{d(d-1)/2})\},\\
 \gnorm{A}{d}&:=
 \inf_{\phi\in\dot W^{1,d}(\R^d)}\norm{A-\nabla\phi}{L^d},\\
 R_d^v(A)&:=\frac{\norm{\nabla\wedge A}{L^{d/2}}}{\gnorm{A}{d}},
 \qquad
 \mathcal Q_d^v(A):=(R_d^v(A))^{d/2}.
\end{align*}
Since mixed weak derivatives commute,
$\nabla\wedge (A-\nabla\phi)=\nabla\wedge A$ in distributions, so the
numerator is gauge invariant. The subspace of $L^d$ gradients is closed
by the distributional Poincar\'e lemma, as in dimension three. Hence
$\gnorm{A}{d}>0$ whenever $\nabla\wedge A\neq0$.
Thus equation~(8) of \cite{FrankLossSharp}
can be written purely in Cartesian terms as
\begin{equation*}
 S_d^v
 =\inf_{\substack{A\in Y_d\\ \nabla\wedge A\neq0}}\mathcal Q_d^v(A).
\end{equation*}
The same infimum is obtained over $C_c^\infty(\R^d;\R^d)$ fields with
nonzero antisymmetric derivative. Indeed, let $A\in Y_d$ and choose a
standard cutoff $\chi_R$, equal to one on $B_R$ and supported in $B_{2R}$,
with $\norm{\nabla\chi_R}{L^d}$ bounded independently of $R$. Then
$\chi_RA\to A$ in $L^d$, and
\[
 \norm{\nabla\chi_R\wedge A}{L^{d/2}}
 \leq\norm{\nabla\chi_R}{L^d}
      \norm{A}{L^d(B_{2R}\setminus B_R)}\longrightarrow0.
\]
The identity
$\nabla\wedge(\chi_RA)=\chi_R\nabla\wedge A+\nabla\chi_R\wedge A$
therefore gives convergence of the antisymmetric derivatives in $L^{d/2}$.
Mollification completes the approximation in the graph norm
$\norm{A}{L^d}+\norm{\nabla\wedge A}{L^{d/2}}$.
Since the gauge seminorm is $1$-Lipschitz in $L^d$, the denominators
also converge, and so do the quotients whenever
$\nabla\wedge A\neq0$.

We next recall the conformal symmetry of the quotient. For $d\geq3$,
a M\"obius transformation of $\R^d\cup\{\infty\}$ is a composition of
translations, orthogonal transformations, positive dilations, and the
inversion $x\mapsto x/\abs{x}^2$, which interchanges $0$ and $\infty$.
Such maps preserve angles wherever they are finite. The pullback below
is the change-of-coordinates rule for a potential viewed as a one-form;
this action is described in \cite[Remark~3]{FrankLossSharp}.
Frank and Loss use the invariance of the gauge quotient under
stereographic projection in \cite[Section~3]{FrankLossVector}; see also
\cite[Sections~2.2 and~6]{WangZhangConstants} for the corresponding
conformal invariance of form quotients. We include the argument on $Y_d$,
including a possible pole of the transformation. It will extend the
strict comparison with one Loss--Yau--Dunne--Min potential to its entire
conformal family.

\begin{Lem}\label{lem:vector-conformal-invariance}
Let $d\geq3$ and let $F$ be a M\"obius transformation of
$\R^d\cup\{\infty\}$. Identifying vector fields with one-forms, define
\[
 (F^*A)(x):=DF(x)^T A(F(x))
\]
away from the possible pole of $F$. Then $F^*$ maps $Y_d$ bijectively
onto itself and preserves both $\norm{\nabla\wedge A}{L^{d/2}}$ and
$\gnorm{A}{d}$. In particular, it preserves $R_d^v$ and $\mathcal Q_d^v$
whenever $\nabla\wedge A\neq0$.
\end{Lem}

\begin{proof}
Write $DF=\lambda O$, with $O$ orthogonal and $\lambda>0$.
Pullback multiplies the pointwise norm of a $k$-form by $\lambda^k$,
whereas $\abs{\det DF}=\lambda^d$. Change of variables therefore
preserves the $L^d$ norm of one-forms and the $L^{d/2}$ norm of two-forms.
Exterior differentiation commutes with pullback away from the pole.
The identity extends distributionally across the pole: with balls
centred there and a cutoff supported in $B_{2\varepsilon}$, equal to
one on $B_\varepsilon$, and satisfying
$\abs{\nabla\chi_\varepsilon}\leq C/\varepsilon$, the additional term
against any fixed smooth test form is bounded by
\[
 C\norm{F^*A}{L^d(B_{2\varepsilon})}
   \norm{\nabla\chi_\varepsilon}{L^{d/(d-1)}}
 \leq C\varepsilon^{d-2}
   \norm{F^*A}{L^d(B_{2\varepsilon})}\longrightarrow0.
\]
Thus a pulled-back $L^d$ gradient remains distributionally closed on
$\R^d$. The distributional Poincar\'e lemma makes it a gradient of a
function in $\dot W^{1,d}(\R^d)$. Applying the same argument to $F^{-1}$
shows that pullback maps the gradient subspace onto itself.
Its $L^d$ isometry consequently preserves distance to that subspace,
which is precisely $\gnorm{A}{d}$.
\end{proof}

When $d=2m+1$ is odd, let $\Sigma$ be the skew-symmetric endomorphism
defined by
\[
 \Sigma e_1=0,\qquad
 \Sigma e_{2k}=e_{2k+1},\qquad
 \Sigma e_{2k+1}=-e_{2k},\qquad k=1,\ldots,m.
\]
The higher-dimensional Loss--Yau potential constructed by Dunne and Min
\cite{DunneMin}, in the formulation of \cite[Appendix~A]{FrankLossZero}, is
\begin{equation}\label{eq:LYDM-potential-d}
 A_{\mathrm{LYDM},d}(x)
 :=\frac{d}{(1+\abs{x}^2)^2}
 \bigl((1-\abs{x}^2)e_1+2x_1x+2\Sigma x\bigr).
\end{equation}
We verify below the known norm identities in
\cite[Lemma~6.1]{WangZhangConstants} in Cartesian coordinates, keeping
track of the two-form normalization in \eqref{eq:antisym-gradient-convention}.

Put
\[
 s:=1+\abs{x}^2,
 \qquad e:=e_1,
 \qquad
 B(x):=(1-\abs{x}^2)e+2x_1x+2\Sigma x,
\]
so that $A_{\mathrm{LYDM},d}=d s^{-2}B$. Since
\[
 \Sigma^T=-\Sigma,\qquad \Sigma e=0,
 \qquad \Sigma^2=-(I-e\otimes e),
\]
one has
\begin{equation}\label{eq:LYDM-basic-identities}
 \abs{B(x)}^2=s^2,
 \qquad x\cdot B(x)=x_1s,
 \qquad \nabla\cdot B(x)=2d x_1.
\end{equation}
The first identity follows by writing $x=x_1e+w$, $w\perp e$, and using
$\abs{\Sigma w}=\abs w$, $w\cdot\Sigma w=0$. Consequently,
\[
 \abs{A_{\mathrm{LYDM},d}}=\frac d s
\]
and, by \eqref{eq:LYDM-basic-identities},
\begin{equation*}
 \nabla\cdot\bigl(\abs{A_{\mathrm{LYDM},d}}^{d-2}
 A_{\mathrm{LYDM},d}\bigr)
 =d^{d-1}\nabla\cdot(s^{-d}B)=0.
\end{equation*}
Hence Lemma~\ref{lem:gauge} yields
$\gnorm{A_{\mathrm{LYDM},d}}{d}=\norm{A_{\mathrm{LYDM},d}}{L^d}$.

For completeness, the numerator can also be computed directly. Let
$J_A=(\partial_jA_i)_{i,j=1}^d$ be the Jacobian matrix and let
$K_A:=J_A-J_A^T$. By \eqref{eq:antisym-gradient-convention},
\begin{equation}\label{eq:K-C-relation}
 \abs{\nabla\wedge A}^2=\frac12\abs{K_A}_{\mathrm F}^2,
\end{equation}
where $\abs{\cdot}_{\mathrm F}$ is the Frobenius norm. Define
\[
 M:=xe^T-ex^T+\Sigma,
 \qquad
 Q:=M-\frac1s(Bx^T-xB^T).
\]
A direct differentiation of $A_{\mathrm{LYDM},d}=ds^{-2}B$ gives
\[
 K_{A_{\mathrm{LYDM},d}}=4d s^{-2}Q.
\]
To compute $Q^TQ$, write $x=te+hu$, where $h>0$, $u\perp e$,
and $\abs u=1$, and put $v=\Sigma u$.
The space $E=\operatorname{span}\{e,u,v\}$ and its orthogonal complement
are invariant under $Q$. In the orthonormal basis $(e,u,v)$ of $E$,
\[
 B=(1+t^2-h^2,\,2th,\,2h),\qquad
 Q|_E=\frac1s
 \begin{pmatrix}
 0&-2h&2th\\
 2h&0&-(1+t^2-h^2)\\
 -2th&1+t^2-h^2&0
 \end{pmatrix}.
\]
The displayed matrix satisfies
$(Q|_E)^TQ|_E=I_E-BB^T/s^2$, since $\abs B=s$.
On $E^\perp$, $Q=\Sigma$ and $\Sigma^T\Sigma=I$.
Combining the two restrictions, and treating $h=0$ by continuity, gives
\begin{equation}\label{eq:Q-projection}
 Q^TQ=I-\frac{BB^T}{s^2}.
\end{equation}
Taking traces in \eqref{eq:Q-projection} and using
\eqref{eq:K-C-relation}, we obtain the pointwise identity
\begin{equation*}
 \abs{\nabla\wedge A_{\mathrm{LYDM},d}}^2
 =8d^2(d-1)s^{-4}.
\end{equation*}
Finally, radial integration and the Euler beta integral give
\begin{equation*}
 \int_{\R^d}(1+\abs{x}^2)^{-d}\dd x
 =\frac{\abs{\Sph^{d-1}}}{2}
   \mathrm B\left(\frac d2,\frac d2\right)
 =2^{-d}\abs{\Sph^d}.
\end{equation*}
Therefore
\begin{align*}
 \int_{\R^d}\abs{A_{\mathrm{LYDM},d}}^d\dd x
 &=d^d2^{-d}\abs{\Sph^d},\\
 \int_{\R^d}\abs{\nabla\wedge A_{\mathrm{LYDM},d}}^{d/2}\dd x
 &=[2d\sqrt{2(d-1)}]^{d/2}2^{-d}\abs{\Sph^d}.
\end{align*}
It follows that
\begin{equation}\label{eq:LYDM-sphere}
 R_d^v(A_{\mathrm{LYDM},d})
 =\sqrt{2(d-1)}\,\abs{\Sph^d}^{1/d},
\end{equation}
and equivalently
\begin{equation*}
 \mathcal Q_d^v(A_{\mathrm{LYDM},d})
 =[2(d-1)]^{d/4}\abs{\Sph^d}^{1/2}.
\end{equation*}

Frank and Loss observed that this field solves the Euler--Lagrange equation
and wrote after equation~(8) that it is conceivable that it is an optimizer
\cite[p.~311]{FrankLossSharp}. Wang and Zhang refer to this as the
Frank--Loss conjecture and prove that its conformal family is not even
locally minimizing in $W^{1,d/2}(\Sph^d;T^*\Sph^d)$, after conformal
identification with one-forms on the round sphere; see
\cite[Theorem~1.3]{WangZhangConstants}. Formula
\eqref{eq:LYDM-potential-d} is available only when $d$ is odd, because the
complex structure $\Sigma$ acts on the even-dimensional space
$e_1^\perp$. 

We now extend the three-dimensional azimuthal competitor to every
dimension $d\geq3$. Put
\begin{equation}\label{eq:vector-trial-d}
 s:=1+\abs{x}^2,
 \qquad \rho^2:=x_1^2+x_2^2,
 \qquad
 A_{\Tclass,d}(x):=\frac{d}{s^2}(-x_2,x_1,0,\ldots,0).
\end{equation}

\begin{Th}\label{thm:FL-vector-all-d}
For every integer $d\geq3$, the field $A_{\Tclass,d}$ belongs to $Y_d$ and
satisfies
\begin{equation}\label{eq:vector-trial-gauge-d}
 \nabla\cdot(\abs{A_{\Tclass,d}}^{d-2}A_{\Tclass,d})=0,
 \qquad
 \gnorm{A_{\Tclass,d}}{d}=\norm{A_{\Tclass,d}}{L^d}.
\end{equation}
Let
\begin{equation}\label{eq:Gd-bd}
 G_d:=\frac{\Gamma(1+d/2)\Gamma((d+1)/2)}{\Gamma(d+1/2)}
     =\frac{2^d}{\binom{2d}{d}},
 \qquad
 b_d:=\frac{2(d-1)}{3d-2}.
\end{equation}
Then
\begin{equation}\label{eq:vector-trial-Ld-d}
 \int_{\R^d}\abs{A_{\Tclass,d}}^d\dd x
 =d^d\,2^{-2d}\abs{\Sph^d}G_d.
\end{equation}
\begin{equation}\label{eq:vector-trial-energy-d}
 \int_{\R^d}\abs{\nabla\wedge A_{\Tclass,d}}^{d/2}\dd x
 =(2d)^{d/2}2^{-d}\abs{\Sph^d}b_d,
\end{equation}
and consequently
\begin{align}
 \mathcal Q_d^v(A_{\Tclass,d})
 &=2^{d/2}\abs{\Sph^d}^{1/2}\frac{b_d}{\sqrt{G_d}},
 \label{eq:vector-trial-powered-d}\\
 R_d^v(A_{\Tclass,d})
 &=2\abs{\Sph^d}^{1/d}
   \left(\frac{4(d-1)^2}{(3d-2)^2G_d}\right)^{1/d}.
 \label{eq:vector-trial-unpowered-d}
\end{align}
Moreover,
\begin{equation}\label{eq:Theta-d}
 \frac{R_d^v(A_{\Tclass,d})}
      {\sqrt{2(d-1)}\abs{\Sph^d}^{1/d}}
 =:\Theta_d
 =\left[
   \frac{4\binom{2d}{d}(d-1)^2}
   {(3d-2)^2[2(d-1)]^{d/2}}
   \right]^{1/d}<1.
\end{equation}
In particular,
\begin{equation*}
 S_d^v
 \leq2^{d/2}\abs{\Sph^d}^{1/2}\frac{b_d}{\sqrt{G_d}},
 \qquad d\geq3,
\end{equation*}
and, for every odd $d\geq3$,
\begin{equation*}
 S_d^v
 \leq\mathcal Q_d^v(A_{\Tclass,d})
 <[2(d-1)]^{d/4}\abs{\Sph^d}^{1/2}
 =\mathcal Q_d^v(A_{\mathrm{LYDM},d}).
\end{equation*}
Hence, in every odd dimension, no nonzero scalar multiple of a conformal
pullback of the Loss--Yau--Dunne--Min potential, with or without an added
$L^d$ gradient, is an optimizer of the Frank--Loss vector inequality.
\end{Th}

\begin{proof}
Write $v=(-x_2,x_1,0,\ldots,0)$. Since
\begin{equation}\label{eq:vector-trial-nonlinear-d}
 \abs{A_{\Tclass,d}}^{d-2}A_{\Tclass,d}
 =d^{d-1}\rho^{d-2}s^{-2(d-1)}v,
\end{equation}
the divergence vanishes away from the axis $\{\rho=0\}$: the scalar
coefficient is constant along planar rotations, while
$\nabla\cdot v=0$. The Cartesian field on the right-hand side of
\eqref{eq:vector-trial-nonlinear-d} is locally $W^{1,1}$ across the axis
for $d\geq3$, so the identity also holds distributionally.
Since $A_{\Tclass,d}(x)=O(\abs{x}^{-3})$ and
$\nabla\wedge A_{\Tclass,d}(x)=O(\abs{x}^{-4})$ as $\abs{x}\to\infty$,
the field belongs to $Y_d$.
Lemma~\ref{lem:gauge}, with $n=q=d$, proves
\eqref{eq:vector-trial-gauge-d}.

For $j\geq3$, direct differentiation gives
\begin{align*}
 F_{12}&=2d\,s^{-3}(s-2\rho^2),\\
 F_{1j}&=-4d\,x_2x_j s^{-3},
 \qquad
 F_{2j}=4d\,x_1x_j s^{-3},
\end{align*}
where throughout this proof
$F_{ij}:=\partial_i(A_{\Tclass,d})_j-\partial_j(A_{\Tclass,d})_i$.
Writing $z=(x_3,\ldots,x_d)$, we obtain
\begin{align}
 \abs{\nabla\wedge A_{\Tclass,d}}^2
 &=(2d)^2s^{-6}\bigl((s-2\rho^2)^2+4\rho^2\abs{z}^2\bigr)=(2d)^2s^{-6}(s^2-4\rho^2).
 \label{eq:vector-trial-pointwise-d}
\end{align}

Let $\Phi:\R^d\to\Sph^d\setminus\{(0,\ldots,0,-1)\}$ be inverse
stereographic projection. In coordinates,
\begin{equation}\label{eq:stereographic-T-d}
\begin{aligned}
 y_j&=\frac{2x_j}{s}\quad(1\leq j\leq d),
 & y_{d+1}&=\frac{1-\abs{x}^2}{s},\\
 \dd\omega_{\Sph^d}&=2^ds^{-d}\dd x,
 & T:=y_1^2+y_2^2&=\frac{4\rho^2}{s^2},
\end{aligned}
\end{equation}
where $x \in \R^d$ and $y \in \Sph^d\setminus\{(0,\ldots,0,-1)\}$. The measure identity is understood under the change of variables
$y=\Phi(x)$. The Jacobian factor $2^ds^{-d}$ follows directly from the derivative matrix
of $\Phi$ (equivalently, from the fact that every tangent length is
multiplied by $2/s$).
Moreover,
$A_{\Tclass,d}^\flat=(d/4)\Phi^*(y_1\dd y_2-y_2\dd y_1)$,
so the trial potential comes from a multiple of the one-form dual to
the standard rotation field in the $(y_1,y_2)$-plane of $\Sph^d$.

For every integrable function $h:[0,1]\to\R$,
\begin{equation}\label{eq:spherical-slicing-d}
 \int_{\Sph^d}h(y_1^2+y_2^2)\dd\omega_{\Sph^d}
 =\frac{d-1}{2}\abs{\Sph^d}
   \int_0^1 h(t)(1-t)^{(d-3)/2}\dd t.
\end{equation}
Indeed, outside a set of surface measure zero we may parameterize $\Sph^d$
by
\[
 y=\bigl(\sqrt t\cos\theta,\sqrt t\sin\theta,
          \sqrt{1-t}\,\omega\bigr),
 \quad 0<t<1,\quad 0\leq\theta<2\pi,
 \quad \omega\in\Sph^{d-2}.
\]
The coordinate vectors in the $t$, $\theta$, and $\omega$ directions are
orthogonal, and the induced surface Jacobian is
\[
 \frac12(1-t)^{(d-3)/2}.
\]
Therefore
\[
 \int_{\Sph^d}h(y_1^2+y_2^2)\dd\omega_{\Sph^d}
 =\pi\abs{\Sph^{d-2}}
  \int_0^1h(t)(1-t)^{(d-3)/2}\dd t.
\]
Since
$\pi\abs{\Sph^{d-2}}=\frac{d-1}{2}\abs{\Sph^d}$,
formula \eqref{eq:spherical-slicing-d} follows.

Applying \eqref{eq:spherical-slicing-d} first with $h(t)=t^{d/2}$ and
then with $h(t)=(1-t)^{d/4}$ gives, purely by one-dimensional integration,
\begin{align}
 \frac1{\abs{\Sph^d}}\int_{\Sph^d}T^{d/2}\dd\omega_{\Sph^d}
 &=\frac{d-1}{2}
   \mathrm B\left(1+\frac d2,\frac{d-1}{2}\right)=\frac{\Gamma(1+d/2)\Gamma((d+1)/2)}{\Gamma(d+1/2)}
 =G_d,
 \notag \\
 \frac1{\abs{\Sph^d}}\int_{\Sph^d}(1-T)^{d/4}\dd\omega_{\Sph^d}
 &=\frac{d-1}{2}
   \int_0^1(1-t)^{(d-3)/2+d/4}\dd t=\frac{2(d-1)}{3d-2}=b_d.
 \label{eq:one-minus-T-integral-d}
\end{align}

Finally, \eqref{eq:gamma-duplication} with $z=(d+1)/2$ gives
\[
 \Gamma(1+d/2)\Gamma((d+1)/2)
 =2^{-d}\sqrt\pi\,\Gamma(d+1).
\]
The same formula with $z=d+1/2$ reads
$\Gamma(d+1/2)\Gamma(d+1)=2^{-2d}\sqrt\pi\,\Gamma(2d+1)$.
Dividing the two identities yields
\begin{equation*}
 G_d
 =\frac{2^d\Gamma(d+1)^2}{\Gamma(2d+1)}
 =\frac{2^d}{\binom{2d}{d}},
\end{equation*}
which explains the second expression in \eqref{eq:Gd-bd}.

Since $\abs{A_{\Tclass,d}}=d\rho s^{-2}$,
the identities $\rho=(s/2)\sqrt T$ and
\eqref{eq:stereographic-T-d} give
\begin{align*}
 \int_{\R^d}\abs{A_{\Tclass,d}}^d\dd x
 &=d^d2^{-d}\int_{\R^d}s^{-d}T^{d/2}\dd x=d^d2^{-2d}\int_{\Sph^d}T^{d/2}\dd\omega_{\Sph^d}=d^d2^{-2d}\abs{\Sph^d}G_d,
\end{align*}
which proves \eqref{eq:vector-trial-Ld-d}.

Likewise, \eqref{eq:vector-trial-pointwise-d} becomes
\[
 \abs{\nabla\wedge A_{\Tclass,d}}^{d/2}
 =(2d)^{d/2}s^{-d}(1-T)^{d/4}.
\]
Thus \eqref{eq:stereographic-T-d} and
\eqref{eq:one-minus-T-integral-d} yield directly
\begin{align*}
 \int_{\R^d}\abs{\nabla\wedge A_{\Tclass,d}}^{d/2}\dd x
 &=(2d)^{d/2}2^{-d}
   \int_{\Sph^d}(1-T)^{d/4}\dd\omega_{\Sph^d}=(2d)^{d/2}2^{-d}\abs{\Sph^d}b_d,
\end{align*}
which proves \eqref{eq:vector-trial-energy-d}.
Dividing the two integrals and using
\eqref{eq:vector-trial-gauge-d} proves
\eqref{eq:vector-trial-powered-d}--\eqref{eq:vector-trial-unpowered-d}.

It remains to prove the strict comparison. Put $Z_d:=\Theta_d^d$. By
\eqref{eq:Gd-bd},
\begin{equation*}
 Z_d=\frac{4\binom{2d}{d}(d-1)^2}
 {(3d-2)^2[2(d-1)]^{d/2}}.
\end{equation*}
The Wallis estimate $\binom{2d}{d}<4^d/\sqrt{\pi d}$ and
$(d-1)/(3d-2)<1/3$ imply
\begin{equation*}
 Z_d<H_d:=\frac{4}{9\sqrt{\pi d}}
 \left(\frac{8}{d-1}\right)^{d/2}.
\end{equation*}
For $d\geq4$,
\begin{equation*}
 \frac{H_{d+1}}{H_d}
 =\sqrt{\frac{8}{d+1}}
   \left(1-\frac1d\right)^{d/2}
 <\sqrt{\frac{8}{e(d+1)}}<1.
\end{equation*}
Also $H_4=128/(81\sqrt\pi)<1$. Hence $Z_d<1$ for all $d\geq4$,
while direct substitution gives $Z_3=40/49<1$. This proves
\eqref{eq:Theta-d}. In odd dimensions, comparison with
\eqref{eq:LYDM-sphere} proves the strict comparison. The conclusion for
the whole family follows from Lemma~\ref{lem:vector-conformal-invariance},
homogeneity of the quotient in $A$, and gauge invariance.
\end{proof}

\begin{Rem}
The field $A_{\Tclass,d}$, its exact quotient, and the strict numerical
comparison in \eqref{eq:Theta-d} are valid also for even $d$. In even
dimensions, however, there is no Loss--Yau--Dunne--Min candidate of the
form considered by Frank and Loss, so the theorem gives an explicit upper
bound for $S_d^v$ rather than the disproof of an optimizer conjecture.
Wang and Zhang had already established nonoptimality of the conformal
Loss--Yau--Dunne--Min family in every odd dimension by their spherical
second-variation argument \cite[Theorem~1.3]{WangZhangConstants}.
On the nonlinear-gauge slice
$\nabla\cdot(\abs{A}^{d-2}A)=0$, their functional is related to our two
normalizations by
\[
 J_{d/2,1}(A)
 =\frac{\left(\int_{\R^d}\abs{\nabla\wedge A}^{d/2}\dd x\right)^2}
        {\int_{\R^d}\abs{A}^d\dd x}
 =(R_d^v(A))^d=(\mathcal Q_d^v(A))^2.
\]
Here Lemma~\ref{lem:gauge} identifies the denominator with
$\gnorm{A}{d}^d$; without the nonlinear-gauge condition, this identity
need not hold. The present theorem supplies a particularly simple
explicit competitor whose quotient is evaluated in closed form.
\end{Rem}

\section{The dimension-three azimuthal field}\label{sec:azimuthal-field}

For $d=3$, the field $A_{\Tclass,3}$ in \eqref{eq:vector-trial-d} is exactly
the field below. We retain the direct three-dimensional computation in
classical curl notation because it also verifies the symmetry needed for
Question~\ref{question:WZ-symmetric}.

Set
\[
 \rho^2:=x_1^2+x_2^2,
 \qquad r^2:=\abs{x}^2=\rho^2+x_3^2,
 \qquad s:=1+r^2,
\]
and consider
\begin{equation*}
 A_{\Tclass}(x):=\frac{3}{s^2}(-x_2,x_1,0).
\end{equation*}
Its relation to the construction of Wang and Zhang is particularly simple.
Define
\[
 A_\pm(x):=\frac{3}{s^2}
 \bigl((1-r^2)e_3+2x_3x\pm2(e_3\wedge x)\bigr).
\]
After a rotation of coordinates, $A_+^\flat$ and $A_-^\flat$ are precisely
the Euclidean representatives of the forms $\xi$ and $\bar\xi$ in
\cite[formulas~(6.4) and~(6.7)]{WangZhangConstants}; $A_+$ is the Loss--Yau potential
with $w=e_3$. Since $e_3\wedge x=(-x_2,x_1,0)$, subtraction gives
\begin{equation}\label{eq:AT-reflected-twin}
 A_{\Tclass}=\frac14(A_+-A_-).
\end{equation}
Thus, up to a nonzero scalar factor, the three-dimensional trial field is
a specific member of their competitor family. We spell out how their
argument already implies strict comparison for this choice.

If $R\in\mathrm{SO}(2)$ acts on the $(x_1,x_2)$-plane and fixes the
$x_3$-axis, then
\begin{equation}\label{eq:AT-equivariance}
 A_{\Tclass}(Rx)=R A_{\Tclass}(x).
\end{equation}
Equivalently, $A_{\Tclass}$ is fixed by the induced action
$A\mapsto R A(R^{-1}\,\cdot)$. Up to a rotation of the distinguished axis,
this is the invariant
class considered in the final question of
\cite{WangZhangConstants}. The same equivariance holds for both $A_+$
and $A_-$, with the same distinguished axis.

To apply their direct comparison, write $\xi$ and $\bar\xi$ for the
corresponding forms on the unit sphere $\Sph^3$. With the stereographic
orientation for which $\operatorname{curl}_{\Sph^3}\xi=2\xi$, one has
\[
 \operatorname{curl}_{\Sph^3}\bar\xi=-2\bar\xi,
 \qquad \abs{\xi}=\abs{\bar\xi}=\frac32.
\]
The spherical inner product, expressed in stereographic coordinates, is
\[
 \ip{\xi}{\bar\xi}_{\Sph^3}
 =\frac94\left(1-\frac{8\rho^2}{s^2}\right),
\]
which is nonconstant. In the proof of
\cite[Proposition~5.1, formula~(5.13)]{WangZhangConstants}, equality holds
precisely when $\ip{\xi}{\varphi_{-1}}\equiv0$; the denominator retains
the infimum over exact forms throughout their estimates~(5.10)--(5.12).
Thus $\varphi_{-1}=-\bar\xi$ yields strict inequality, and
$J_2(\xi-\bar\xi)<J_2(\xi)$ on the sphere.
Conformal invariance and homogeneity, together with
\eqref{eq:AT-reflected-twin}, give
$J_2(A_{\Tclass})<J_2(A_{\mathrm{LY}})$.
Thus the negative answer to Question~\ref{question:WZ-symmetric} follows
already from their comparison after this choice of competitor and the
symmetry check. The calculation below additionally gives the exact
quotient by working directly in Euclidean coordinates.

Since
\[
 \abs{A_{\Tclass}(x)}=\frac{3\rho}{s^2},
 \qquad
 \abs{A_{\Tclass}}A_{\Tclass}
 =9\rho s^{-4}(-x_2,x_1,0),
\]
put $f(x):=9\rho s^{-4}$. On the set $\{\rho>0\}$ the function $f$
depends on $(x_1,x_2)$ only through $\rho$, and therefore
\[
 \nabla\cdot\bigl(f(-x_2,x_1,0)\bigr)
 =-x_2\partial_1f+x_1\partial_2f=0.
\]
Moreover, $f(-x_2,x_1,0)$ extends across $\{\rho=0\}$ as a locally
$W^{1,1}$ vector field (in fact its first derivatives are locally bounded
near the axis). Hence the almost-everywhere identity extends to the sense
of distributions, and
\begin{equation*}
 \nabla\cdot(\abs{A_{\Tclass}}A_{\Tclass})=0.
\end{equation*}
Lemma~\ref{lem:gauge} then gives
$\gnorm{A_{\Tclass}}{3}=\norm{A_{\Tclass}}{3}$.

Since $\abs{A_{\Tclass}}=3\rho/s^2$, spherical coordinates yield
\begin{align}
 \norm{A_{\Tclass}}{3}^3
 &=27\int_{\R^3}\frac{\rho^3}{(1+r^2)^6}\dd x=54\pi
 \left(\int_0^\pi\sin^4\vartheta\dd\vartheta\right)
 \left(\int_0^\infty\frac{r^5}{(1+r^2)^6}\dd r\right)
 =\frac{27\pi^2}{80}.
 \label{eq:AT-L3}
\end{align}
A direct computation gives
\begin{equation*}
 \nabla\wedge A_{\Tclass}
 =\frac{6}{s^3}
 (2x_1x_3,2x_2x_3,1+x_3^2-\rho^2).
\end{equation*}
Moreover,
\[
 4x_3^2\rho^2+(1+x_3^2-\rho^2)^2=s^2-4\rho^2,
\]
and hence
\begin{equation*}
 \abs{\nabla\wedge A_{\Tclass}}
 =\frac{6}{s^3}\sqrt{s^2-4\rho^2}.
\end{equation*}
Therefore
\begin{equation}\label{eq:J-definition}
 \norm{\nabla\wedge A_{\Tclass}}{3/2}^{3/2}
 =6^{3/2}J,
 \qquad
 J:=\int_{\R^3}\frac1{s^3}
 \left(1-\frac{4\rho^2}{s^2}\right)^{3/4}\dd x.
\end{equation}

We now compute $J$ by a direct change of variables in $\R^3$, obtained
from the standard torus coordinates on $\Sph^3$ by stereographic projection. For
$0<t<1$ and $0\leq\theta,\varphi<2\pi$, set
\[
 D(t,\varphi):=1+\sqrt{1-t}\,\sin\varphi
\]
and
\begin{equation}\label{eq:analytic-change-J}
 x_1=\frac{\sqrt t\cos\theta}{D(t,\varphi)},
 \qquad
 x_2=\frac{\sqrt t\sin\theta}{D(t,\varphi)},
 \qquad
 x_3=\frac{\sqrt{1-t}\cos\varphi}{D(t,\varphi)}.
\end{equation}
Apart from sets of Lebesgue measure zero, this parametrizes $\R^3$ once.
Indeed, if $\rho>0$ and
\[
 t:=\frac{4\rho^2}{s^2},
\]
then $0<t<1$ except on a null set, while $\theta$ is the polar angle of
$(x_1,x_2)$ and $\varphi$ is determined by
\[
 \cos\varphi=\frac{2x_3}{s\sqrt{1-t}},
 \qquad
 \sin\varphi=\frac{1-r^2}{s\sqrt{1-t}}.
\]
The last two quantities have squares summing to one because
\[
 4x_3^2+(1-r^2)^2=s^2-4\rho^2=s^2(1-t).
\]
For the variables in \eqref{eq:analytic-change-J}, a direct calculation
gives
\[
 r^2=\frac{1-\sqrt{1-t}\sin\varphi}
          {1+\sqrt{1-t}\sin\varphi},
 \qquad
 s=1+r^2=\frac{2}{D(t,\varphi)},
 \qquad
 \frac{4\rho^2}{s^2}=t.
\]
To compute the Jacobian explicitly, write
\[
 R(t,\varphi):=\frac{\sqrt t}{D(t,\varphi)},
 \qquad
 Z(t,\varphi):=\frac{\sqrt{1-t}\cos\varphi}{D(t,\varphi)},
\]
so that $(x_1,x_2,x_3)=(R\cos\theta,R\sin\theta,Z)$. Hence
\[
 \left|
 \det\frac{\partial(x_1,x_2,x_3)}
          {\partial(t,\theta,\varphi)}
 \right|
 =R\left|R_tZ_\varphi-R_\varphi Z_t\right|.
\]
Put $a:=\sqrt t$ and $b:=\sqrt{1-t}$, so $D=1+b\sin\varphi$.
Direct differentiation gives
\[
 R_t=\frac{b+\sin\varphi}{2abD^2},
 \qquad
 R_\varphi=-\frac{ab\cos\varphi}{D^2},
 \qquad
 Z_t=-\frac{\cos\varphi}{2bD^2},
 \qquad
 Z_\varphi=-\frac{b(b+\sin\varphi)}{D^2}.
\]
Therefore
\begin{align*}
 R_tZ_\varphi-R_\varphi Z_t
 &=-\frac{(b+\sin\varphi)^2+a^2\cos^2\varphi}
          {2aD^4}
 =-\frac{1}{2aD^2},
\end{align*}
because $a^2+b^2=1$. Since $R=a/D$, we have proved
\begin{equation*}
 \left|
 \det\frac{\partial(x_1,x_2,x_3)}
          {\partial(t,\theta,\varphi)}
 \right|
 =\frac{1}{2D(t,\varphi)^3}.
\end{equation*}
Consequently,
\[
 s^{-3}\dd x
 =\frac{D(t,\varphi)^3}{8}
  \frac{1}{2D(t,\varphi)^3}
  \dd t\dd\theta\dd\varphi
 =\frac1{16}\dd t\dd\theta\dd\varphi.
\]
Substituting this into \eqref{eq:J-definition}, we obtain the elementary
one-dimensional integral
\begin{align*}
 J
 &=\frac1{16}\int_0^{2\pi}\int_0^{2\pi}\int_0^1
   (1-t)^{3/4}\dd t\dd\theta\dd\varphi=\frac{\pi^2}{4}\int_0^1(1-t)^{3/4}\dd t
 =\frac{\pi^2}{4}\cdot\frac47
 =\frac{\pi^2}{7}.
\end{align*}
Thus the value of $J$ follows from an explicit change of variables and
a Jacobian calculation in $\R^3$, independently of the spherical
integration formula used in Section~\ref{sec:FL-vector-all-d}.
Consequently,
\begin{equation}\label{eq:AT-curl-norm}
 \norm{\nabla\wedge A_{\Tclass}}{3/2}
 =6\left(\frac{\pi^2}{7}\right)^{2/3}.
\end{equation}
Combining \eqref{eq:AT-L3} and \eqref{eq:AT-curl-norm}, we obtain
\begin{equation}\label{eq:AT-quotient}
\begin{aligned}
 \frac{\norm{\nabla\wedge A_{\Tclass}}{3/2}}
      {\gnorm{A_{\Tclass}}{3}}
 &=2\left(\frac{80\pi^2}{49}\right)^{1/3}=2\left(\frac{40}{49}\right)^{1/3}\abs{\Sph^3}^{1/3}
 <2\abs{\Sph^3}^{1/3}.
\end{aligned}
\end{equation}
Thus $C_v^{\Tclass}$ satisfies the upper bound in
\eqref{eq:main-Cv}. Raising \eqref{eq:AT-quotient} to the power $3/2$ gives
\begin{equation*}
 \frac{\norm{\nabla\wedge A_{\Tclass}}{3/2}^{3/2}}
      {\gnorm{A_{\Tclass}}{3}^{3/2}}
 =\frac{8\sqrt{10}}7\pi<4\pi.
\end{equation*}
In view of \eqref{eq:AT-equivariance}, $A_{\Tclass}$ is fixed by $G$.
Moreover, \eqref{eq:AT-quotient} and \eqref{eq:FL-candidates} give directly
\begin{align*}
 J_2(A_{\Tclass})
 =4\left(\frac{80\pi^2}{49}\right)^{2/3}=\left(\frac{40}{49}\right)^{2/3}
   J_2(A_{\mathrm{LY}})<J_2(A_{\mathrm{LY}}).
\end{align*}
Together with \eqref{eq:LY-equivariance}, this gives a negative answer to
Question~\ref{question:WZ-symmetric}.
This proves the vector-field part of Theorem~\ref{thm:main}.

\section{Consequences and final remarks}

As noted above, Wang and
Zhang had already obtained the corresponding nonoptimality statements by
direct comparison and instability arguments
\cite{WangZhangSpinor,WangZhangConstants}.
Theorem~\ref{thm:FL-spinor-all-d} gives a particularly simple alternative for
the spinorial conjecture in \cite{FrankLossSharp}. In
$d=3$, the conformal spinors in that family all have quotient
$3\sqrt3\pi/2$ in the power normalization, whereas the constant-direction
Aubin--Talenti spinor has the strictly smaller quotient $\sqrt{18\pi}$.
Similarly, Theorem~\ref{thm:FL-vector-all-d} gives a single explicit
formula for a vector competitor in every $d\geq3$. It disproves the
Loss--Yau--Dunne--Min optimizer conjecture in every odd dimension, while in
even dimensions it gives the quantitative upper bound
\eqref{eq:vector-trial-powered-d} for $S_d^v$.
The estimates also imply
\begin{align*}
 C_s C_v
 &\leq 2\left(\frac{1440\pi^3}{49}\right)^{1/3}=\left(\frac{320}{147\pi}\right)^{1/3}
 3\abs{\Sph^3}^{2/3}
 <3\abs{\Sph^3}^{2/3}.
\end{align*}
Since the Sobolev constant in the normalization of \cite{FrankLossZero} is
\[
 S_3=\frac34\abs{\Sph^3}^{2/3},
\]
the value $3\abs{\Sph^3}^{2/3}$ equals $4S_3$.
The strategy discussed in \cite[Remark~9.1]{FrankLossZero} combines
the two separate inequalities. The strict upper bound above shows that
their optimal constants alone cannot recover the sharp zero-mode
threshold $4S_3$ by this strategy.

It is important to distinguish the gauge quotient $J_2$ in
\eqref{eq:WZ-J2} from the helicity quotient treated sharply in
\cite[Corollary~1.2]{WangZhangSharp}. For fields of positive helicity, Wang and Zhang prove
\begin{equation*}
 \frac{\left(\displaystyle\int_{\R^3}
       \abs{\nabla\wedge A}^{3/2}\dd x\right)^{4/3}}
      {\displaystyle\int_{\R^3}A\cdot(\nabla\wedge A)\dd x}
 \geq2\abs{\Sph^3}^{1/3},
\end{equation*}
with equality for nonzero scalar multiples of the Loss--Yau potential, up to
orientation-preserving conformal and gauge transformations.
This result neither determines $C_v$ nor contradicts the strict
inequality for $J_2$.

For smooth zero-mode pairs on the round sphere $\Sph^3$, Wang and Zhang
also prove a sharp magnetic-field bound
\cite[Theorem~1.6 and Section~8]{WangZhangSharp}.
More precisely, let $a$ be a smooth real vector field and let
$\varphi\not\equiv0$ be a smooth spinor satisfying the magnetic Dirac
equation on $\Sph^3$. Let $A^\flat$ be the pullback of $a^\flat$ to
$\R^3$ by inverse stereographic projection, with the corresponding
conformal transformation of the spinor. Conformal covariance then gives
\begin{equation*}
 \norm{\nabla\wedge A}{3/2}
 \geq3\abs{\Sph^3}^{2/3}=4S_3.
\end{equation*}
Equality is attained by the Killing-spinor/Reeb-field pairs described in
their theorem and their conformal and gauge transforms. Thus the sharp
threshold in this conformally compactifiable class does not arise by
multiplying the separate optimal constants $C_s$ and $C_v$.
In their subsequent work
\cite[Theorem~1.1 and Section~5.5]{WangZhangCriterion}, Wang and Zhang
establish lower bounds for the magnetic field of Dirac zero modes on
$\Sph^d$, $d\geq3$, assuming that the spinor belongs to $L^p$ for some
$d/(d-1)<p<\infty$ and that $\exd a^\flat\in L^{d/2}$.
The bound in their Theorem~1.1 is attained in odd dimensions; a stronger
estimate in dimension four is given in their Theorem~5.3.
These results concern the coupled magnetic Dirac equation and do not
determine the separate Sobolev constants $S_d^s$ and $S_d^v$.
No extension of the preceding spherical bounds to arbitrary weak Euclidean
zero-mode pairs is needed for the results proved here.

\medskip
\noindent\textbf{Acknowledgments.}
J. Mederski was partly supported by the National Science Centre, Poland,
under grant no.~2023/51/B/ST1/00968.

\medskip
\noindent\textbf{AI assistance statement.}
The authors used AI tools to assist with manuscript editing and checking
mathematical arguments developed by the authors. All proofs were independently
verified by the authors, who take full responsibility for the mathematical
correctness and final content of the manuscript.

\medskip
\noindent\textbf{Data availability statement.} Not applicable.

\medskip
\noindent\textbf{Conflict of interest statement}. On behalf of all authors, the corresponding author states that there is no conflict of interest.

\end{document}